\documentclass[11pt]{article}
\usepackage{amsmath,amssymb,amsthm,color}

\usepackage{comment}
\usepackage{color}
\usepackage{amssymb}
\usepackage{amsmath}
\usepackage[english]{babel}
\usepackage{fancyhdr}
\usepackage{amsmath}

\newtheorem{thm}{Theorem}[section]

\newtheorem{lemma}{Lemma}[section]

\newtheorem{cor}{Corollary}[section]
\newtheorem{defn}{Definition}[section]

\def\R{{\mathfrak R}\, }

\def\ci{\begin{color}{red}\,}
\def\cf{\end{color}\,}

\begin{document}
\begin{center}
{\Large\bf Generalized Jordan derivations on semiprime rings}

\vspace{.2in}
{\bf Bruno L. M. Ferreira\textsuperscript{1}} 
\\
{\bf Henrique Guzzo Jr\textsuperscript{2}}
\\
and
\\
{\bf Ruth N. Ferreira\textsuperscript{3}}
\vspace{.2in}

 Universidade Tecnol\'{o}gica Federal do Paran\'{a}\textsuperscript{1,3}, Avenida Professora Laura Pacheco Bastos, 800, 85053-510, Guarapuava, Brazil.
\\
and
\\
 Universidade de S\~{a}o Paulo\textsuperscript{2}, Rua do Mat\~{a}o, 1010, 05508-090, S\~{a}o Paulo, Brazil.
\vspace{.2in}
 
brunoferreira@utfpr.edu.br
\\
guzzo@ime.usp.br
\\
and
\\
ruthnascimento@utfpr.edu.br
\vspace{.2in}

\end{center}

\begin{abstract}
The purpose of this note is to prove the following. Suppose $\R$ is a
semiprime unity ring having an idempotent element e $\left(e \neq 0, e \neq 1\right)$ which satisfies mild conditions. It is shown that every additive generalized Jordan derivation on $\R$ is a generalized derivation.
\end{abstract}


{\bf keywords:} Jordan derivation; Generalized Jordan derivation; Rings

\section{Introduction}
Let $\R$ be a ring, recall that an additive (linear) map $\delta$ from $\R$ in itself is called a derivation if $\delta(ab) = \delta(a)b+a\delta(b)$
for all $a,b \in \R$; a Jordan derivation if $\delta(a^2) = \delta(a)a+a\delta(a)$ for each $a \in \R$; and a Jordan
triple derivation if $\delta(aba) = \delta(a)ba+a\delta(b)a+ab\delta(a)$ for all $a,b \in \R$. More generally, if
there is a derivation $\tau : \R \rightarrow \R$ such that $\delta(ab) = \delta(a)b +a\tau (b)$ for all $a,b \in \R$, then $\delta$ is
called a generalized derivation and $\tau$ is the relating derivation; if there is a Jordan derivation
$\tau : \R \rightarrow \R$ such that $\delta(a^2) = \delta(a)a + a\tau(a)$ for all $a \in \R$, then $\delta$ is called a generalized
Jordan derivation and $\tau$ is the relating Jordan derivation. 
The structures of derivations, Jordan derivations, generalized derivations and generalized
Jordan derivations were systematically studied. It is obvious that every generalized derivation
is a generalized Jordan derivation and every derivation is a Jordan derivation. But the converse is in general not true.
Herstein \cite{Herstein} showed that every Jordan derivation from a $2$-torsion free prime ring into itself is a derivation.
Bre$\check{s}$ar \cite{Bresar} proved that Herstein’s result is true for $2$-torsion free semiprime rings.
Jing and Lu, motivated by the concept of generalized derivation, initiate this concept of generalized Jordan derivation in \cite{Jing}. Moreover in \cite{Jing} the authors conjecture that every generalized Jordan derivation on $2$-torsion free semiprime ring is a generalized derivation.

In the present paper we characterized generalized Jordan derivation on a semiprime ring $\R$. We prove that if there is a nontrivial idempotent element in $\R$ which satisfies mild conditions, then every additive generalized Jordan derivation is an additive generalized derivation.

In the ring $\R$, let $e$ be an idempotent element so that $e \neq 0$, $e \neq 1$. As in \cite{Jacobson}, the two-sided Peirce decomposition of $\R$ relative to the
idempotent e takes the form $\R = e\R e \oplus e\R(1-e)\oplus (1-e)\R e \oplus (1-e)\R(1-e)$. We will formally
set $e_1= e$ and $e_2= 1-e$. So letting $\R_{mn} = e_m \R e_n$; $m,n = 1,2$ we may write $\R = \R_{11} \oplus \R_{12} \oplus\R_{21} \oplus \R_{22}$. Moreover an element of the subring $\R_{mn}$ will be denoted by $a_{mn}$. 

\section{Results and proofs}

In this section, we discuss the additive generalized Jordan derivations on rings. The
following is our main result:

\begin{thm}\label{main}
Let $\R$ be a $2$-torsion free semiprime unity ring and $\delta$ be an additive generalized Jordan
derivation from $\R$ into itself. If there exist an idempotent $e$ so that $e \neq 0$, $e \neq 1$ in $\R$, then $\delta$ is an additive generalized derivation.
\end{thm}

Henceforth let $\R$ be a $2$-torsion free semiprime unity ring containing a nontrivial idempotent $e_1$.
Consider $\R = \R_{11} \oplus \R_{12} \oplus\R_{21} \oplus \R_{22}$ the Peirce decomposition relative to the idempotent $e_1$ satisfying the following conditions
\begin{enumerate}
\item [($\spadesuit$)] If $x\cdot\R_{12} = 0$ then $e_1 x e_1 = 0$ and $e_2 x e_1 = 0$, 
\item [($\clubsuit$)] If $x\cdot\R_{22} = 0$ then $e_1 x e_2 = 0$ and $e_2 x e_2 = 0$,
 \end{enumerate}
$\delta: \R \rightarrow \R$ is an additive generalized Jordan derivation and
$\tau: \R \rightarrow \R$ is the relating Jordan derivation such that $\delta(a^2) = \delta(a)a + a\tau(a)$ for all $a \in \R$. We will complete the proof of above theorem by proving several lemmas.

\begin{lemma}\label{lema1}
For all $a, b, c \in \R$, the following statements hold:
\begin{enumerate}
\item[\it (i)] $\delta(ab + ba) = \delta(a)b + a\tau(b) + \delta(b)a + b\tau(a)$;
\item[\it (ii)] $\delta(aba) = \delta(a)ba + a\tau(b)a + ab\tau(a)$;
\item[\it (iii)] $\delta(abc + cba) = \delta(a)bc + a\tau(b)c + ab\tau(c) + \delta(c)ba + c\tau(b)a + cb\tau(a)$.
\end{enumerate}
\end{lemma}
\begin{proof}
See Lemma $2.1$ in \cite{Jing}.
\end{proof}
\begin{lemma}\label{lema2}
$\tau(e_1) = [e_1, s]$ for some $s \in \R$, where $[x,y] = xy - yx$ for $x, y \in \R$.
\end{lemma}
\begin{proof}
Write $\tau(e_1) = s_{11} + s_{12} + s_{21} +  s_{22}$. Since $\tau(e_1) = \tau(e_1)e_1 + e_1\tau(e_1)$, we have $s_{11} + s_{12} + s_{21}+
s_{22} = 2s_{11} + s_{12} + s_{21}$, which implies that $s_{11} = s_{22} = 0$ and $\tau(e_1) = s_{12} + s_{21}$. Let $s = s_{12} - s_{21}$. It is obvious
that $\tau(e_1) = [e_1, s]$.
\end{proof}

Observe that $d_s: \R \rightarrow \R$ so that $d_s(a) = [a,s]$ is a derivation and thus a Jordan derivation. Define $\Delta$ by $\Delta(a) = \delta(a) - d_{s}(a)$ for each $a \in \R$. Clearly, $\Delta$ is also an additive generalized Jordan derivation from $\R$ into itself, and $\Xi : \R \rightarrow \R$ defined by $\Xi(a) = \tau(a) - d_s(a)$ for each $a \in \R$ is the relating additive Jordan derivation.
Note that $\Xi(e_1) = \Xi(e_2) = 0$. 

\begin{lemma}\label{lema3}
$\Xi(a_{ij}) \subset \R_{ij}$ for any $a_{ij} \in \R_{ij}$ $(i, j = 1, 2)$.
\end{lemma}
\begin{proof}
\textbf{Case $1$:} For $i=j=1$, $a_{11}=e_1 a_{11}e_1$ we have from Lemma \ref{lema1} item $(ii)$
$$\Xi(a_{11})= \Xi(e_1 a_{11}e_1) = \Xi(e_1) a_{11}e_1 + e_1 \Xi(a_{11})e_1 + e_1 a_{11}\Xi(e_1) = e_1 \Xi(a_{11})e_1.$$
Hence $\Xi(a_{11}) \in \R_{11}$.

\textbf{Case $2$:} For $i=j=2$ write $\Xi(a_{22}) = b_{11} + b_{12} + b_{21} + b_{22}$ we have from Lemma \ref{lema1} item $(i)$
\begin{eqnarray*}
0=\Xi(e_1 a_{22} + a_{22}e_1) &=& \Xi(e_1)a_{22} + e_1\Xi(a_{22}) + \Xi(a_{22})e_1 + a_{22}\Xi(e_1) \\&=& e_1\Xi(a_{22}) + \Xi(a_{22})e_1 = 2b_{11} + b_{12} + b_{21}.
\end{eqnarray*}
Thus, $\Xi(a_{22}) \in \R_{22}$.

\textbf{Case $3$:} For $i=1$ and $j=2$ write $\Xi(a_{12}) = b_{11} + b_{12} + b_{21} + b_{22}$ we have from Lemma \ref{lema1} items $(i), (ii)$
$$\Xi(a_{12}) = \Xi(e_1a_{12} + a_{12}e_1) = e_1\Xi(a_{12})$$
and
$$0 = \Xi(e_1 a_{12} e_1) = e_1 \Xi(a_{12})e_1.$$
Hence, $\Xi(a_{12}) \in \R_{12}$.

\textbf{Case $4$:} Finally, for $i=2$ and $j=1$ write $\Xi(a_{21}) = b_{11} + b_{12} + b_{21} + b_{22}$ we have from Lemma \ref{lema1} items $(i), (ii)$
$$\Xi(a_{21}) = \Xi(e_1a_{21} + a_{21}e_1) = \Xi(a_{21})e_1$$
and
$$0 = \Xi(e_1 a_{21} e_1) = e_1 \Xi(a_{21})e_1.$$
Thus, $\Xi(a_{21}) \in \R_{21}$.
\end{proof}

\begin{lemma}\label{lema5}
$\Delta(a_{ij}) \subset \R_{ij} + \R_{jj}$. 
\end{lemma}
\begin{proof}
Firstly,we prove that $\Delta(e_1) \in \R_{11} + \R_{21}$. Let $\Delta(e_1) = a_{11} + a_{12} + a_{21} + a_{22}$. Since $\Delta(e_1) = \Delta(e_1)e_1 +e_1\Xi(e_1) = \Delta(e_1)e_1$, we see that $a_{11} + a_{12} + a_{21}+ a_{22} = a_{11} + a_{21}$, which implies that $a_{12} = a_{22} = 0$ and $\Delta(e_1) = a_{11} + a_{21} \in \R_{11} + \R_{21}$.

\textbf{Case $1$:} For $i=1$ and $j=2$ let $a_{12} \in \R_{12}$ and $\Delta(a_{12}) = b_{11} + b_{12} + b_{21} + b_{22}$. Then 
\begin{eqnarray*}
b_{11} + b_{12} + b_{21} + b_{22} &=& \Delta(a_{12}) \\&=& \Delta(e_1 a_{12} + a_{12}e_1) \\&=& \Delta(e_1)a_{12} + e_1\Xi(a_{12}) + \Delta(a_{12})e_1 + a_{12}\Xi(e_1)\\&=& \Delta(e_1)a_{12} + \Xi(a_{12}) + b_{11} + b_{21}. 
\end{eqnarray*}
Hence $b_{12} + b_{22} = \Delta(e_1)a_{12} + \Xi(a_{12}) \in \R_{12} + \R_{22}$.
On the other hand, 
\begin{eqnarray*}
b_{11} + b_{12} + b_{21} + b_{22} &=& \Delta(a_{12})= \Delta(a_{12}e_2 + e_2a_{12})\\&=& \Delta(a_{12})e_2 + a_{12}\Xi(e_2)+\Delta(e_2)a_{12} + e_2\Xi(a_{12})\\&=& \Delta(a_{12})e_2 + \Delta(e_2)a_{12}\\&=& b_{12} + b_{22} + \Delta(e_2)a_{12}.
\end{eqnarray*}
Thus we get $b_{11} + b_{12} + b_{21} + b_{22} = \Delta(e_1)a_{12} + \Xi(a_{12}) + \Delta(e_2)a_{12}$ , which implies that $\Delta(a_{12}) \in \R_{12} + \R_{22}$.

\textbf{Case $2$:} For $i=2$ and $j=1$ let $a_{21} \in \R_{21}$ and $\Delta(a_{21}) = b_{11} + b_{12} + b_{21} + b_{22}$. Then
\begin{eqnarray*}
 b_{11} + b_{12} + b_{21} + b_{22} &=& \Delta(a_{21}) \\&=& \Delta(a_{21} e_1 + e_1 a_{21}) \\&=& \Delta(a_{21})e_1 + a_{21}\Xi(e_1) + \Delta(e_{1})a_{21} + e_1\Xi(a_{21}) \\&=& b_{11} + b_{21}.
 \end{eqnarray*} 
Therefore $\Delta(a_{21}) \in \R_{11} + \R_{21}$.
\end{proof}

\begin{lemma}\label{lema6}
$\Delta(a_{ii}) \subset \R_{ii} + \R_{ji}$ with $i\neq j$. 
\end{lemma}
\begin{proof}
\textbf{Case $1$:} For $i= 1$, by (ii) of Lemma \ref{lema1} we have 
\begin{eqnarray*}
\Delta(a_{11}) &=& \Delta(e_1 a_{11} e_1) \\&=& \Delta(e_1)a_{11}e_1 + e_1\Xi(a_11)e_1 + e_1a_{11}\Xi(e_1) \\&=& \Delta(e_1)a_{11} +\Xi(a_11).
\end{eqnarray*}
Therefore $\Delta(a_{11}) \in \R_{11} + \R_{21}.$

\textbf{Case $2$:} The proof is similar to Case $1$.
\end{proof}

\begin{lemma}\label{lema7}
\begin{enumerate}
\item[(1)] $\Delta(a_{11}b_{12}) = \Delta(a_{11})b_{12} + a_{11}\Xi(b_{12})$ holds for all $a_{11} \in \R_{11}$ and $b_{12} \in \R_{12}$.
\item[(2)] $\Delta(a_{12}b_{22}) = \Delta(a_{12})b_{22} + a_{12}\Xi(b_{22})$ holds for all $a_{12} \in \R_{12}$ and $b_{22} \in \R_{22}$.
\item[(3)] $\Delta(a_{21}b_{12}) = \Delta(a_{21})b_{12} + a_{21}\Xi(b_{12})$ holds for all $a_{21} \in \R_{21}$ and $b_{12} \in \R_{12}$.
\item[(4)] $\Delta(a_{22}b_{22}) = \Delta(a_{22})b_{22} + a_{22}\Xi(b_{22})$ holds for all $a_{22},b_{22} \in \R_{22}$.
\end{enumerate}
\end{lemma}
\begin{proof}
For any $a_{11} \in \R_{11}$ and $b_{12} \in \R_{12}$, it follows from Lemma \ref{lema1} and Lemma \ref{lema5} that
\begin{eqnarray*}
\Delta(a_{11}b_{12}) &=& \Delta(a_{11}b_{12} + b_{12}a_{11}) \\&=& \Delta(a_{11})b_{12} + a_{11}\Xi(b_{12}) + \Delta(b_{12})a_{11} + b_{12}\Xi(a_{11})\\&=& \Delta(a_{11})b_{12} + a_{11}\Xi(b_{12}).
\end{eqnarray*}
Similarly, $(2)$ is true for all $a_{12} \in \R_{12}$ and $b_{22} \in \R_{22}$.

Now for any $a_{21} \in \R_{21}$ and $b_{12} \in \R_{12}$, it follows from Lemma \ref{lema1} and Lemma \ref{lema5} that
\begin{eqnarray*}
\Delta(a_{21}b_{12}) &=& \Delta(e_2a_{21}b_{12}e_2) \\&=& \Delta(e_2)a_{21}b_{12} + e_2\Xi(a_{21}b_{12})e_2 + e_2 a_{21}b_{12}\Xi(e_{2}) \\&=& (\Delta(a_{21}) - \Xi(a_{21}))b_{12} + \Xi(a_{21}b_{12}) \\&=& \Delta(a_{21})b_{12} + a_{21}\Xi(b_{12}).
\end{eqnarray*}
Finally for any $a_{22} \in \R_{22}$, by Lemma \ref{lema1} item $(ii)$, we have
\begin{eqnarray*}
\Delta(a_{22}) &=& \Delta(e_2a_{22}e_2)\\&=& \Delta(e_2)a_{22}e_2 + e_2\Xi(a_{22})e_2 + e_2a_{22}\Xi(e_2)\\&=& \Delta(e_2)a_{22} + \Xi(a_{22}),
\end{eqnarray*}
and hence $\Delta(a_{22}b_{22}) = \Delta(e_2)a_{22}b_{22} + \Xi(a_{22}b_{22})$ holds for all $a_{22},b_{22} \in \R_{22}$. Since 
\begin{eqnarray*}
\Delta(a_{22})b_{22} + a_{22}\Xi(b_{22}) &=& \Delta(e_2)a_{22}b_{22} + \Xi(a_{22})b_{22} + a_{22}\Xi(b_{22})\\&=& \Delta(e_2)a_{22}b_{22} + \Xi(a_{22}b_{22}),
\end{eqnarray*}
we get that $\Delta(a_{22}b_{22}) = \Delta(a_{22})b_{22} + a_{22}\Xi(b_{22})$.
\end{proof}

\begin{lemma}\label{lema8}
$\Delta(ab) = \Delta(a)b + a\Xi(b)$ for all $a,b \in \R$, that is, $\Delta$ is an additive generalized derivation.
\end{lemma}
\begin{proof}
$\bullet$ First : For any $a,b \in \R$ and $x_{12} \in \R_{12}$, by Lemmas \ref{lema1}-\ref{lema7}, we have
\begin{eqnarray*}
\Delta(abx_{12}) &=& \Delta(a_{11}b_{11}x_{12} + a_{12}b_{21}x_{12} + a_{22}b_{21}x_{12} + a_{21}b_{11}x_{12}) \\&=& \Delta(a_{11}b_{11})x_{12} + a_{11}b_{11}\Xi(x_{12}) +  \Delta(a_{12}b_{21})x_{12} + a_{12}b_{21}\Xi(x_{12}) \\&+& \Delta(a_{22}b_{21})x_{12} + a_{22}b_{21}\Xi(x_{12}) + \Delta(a_{21}b_{11})x_{12} + a_{21}b_{11}\Xi(x_{12}) \\&=& \Delta(a_{11}b_{11} + a_{12}b_{21} + a_{22}b_{21} + a_{21}b_{11})x_{12} \\&+& (a_{11}b_{11} + a_{12}b_{21} + a_{22}b_{21} + a_{21}b_{11})\Xi(x_{12}) \\&=&
\Delta(ab)x_{12} + ab\Xi(x_{12}).
\end{eqnarray*}

$\bullet$ Second: For any $x_{12} \in \R_{12}$, by Lemmas \ref{lema1}-\ref{lema7}, we get
\begin{eqnarray*}
\Delta(abx_{12}) &=& \Delta(a_{11}b_{11}x_{12} + a_{12}b_{21}x_{12} + a_{22}b_{21}x_{12} + a_{21}b_{11}x_{12}) \\&=& \Delta(a_{11})b_{11}x_{12} + a_{11}\Xi(b_{11}x_{12}) +  \Delta(a_{12})b_{21}x_{12} + a_{12}\Xi(b_{21}x_{12}) \\&+& \Delta(a_{22})b_{21}x_{12} + a_{22}\Xi(b_{21}x_{12}) + \Delta(a_{21})b_{11}x_{12} + a_{21}\Xi(b_{11}x_{12}) \\&=& \Delta(a)bx_{12} + a\Xi(b)x_{12} + ab\Xi(x_{12}).
\end{eqnarray*}
So $(\Delta(ab) - \Delta(a)b - a\Xi(b))x_{12} = 0$ for any $x_{12} \in \R_{12}$. Hence $e_1(\Delta(ab) - \Delta(a)b - a\Xi(b))e_1 = 0 = e_2(\Delta(ab) - \Delta(a)b - a\Xi(b))e_1$, by condition ($\spadesuit$).

$\bullet$ Third: For any $x_{22} \in \R_{22}$, we compute $\Delta(abx_{22})$.
\begin{eqnarray*}
\Delta(abx_{22}) &=& \Delta(a_{11}b_{12}x_{22}) + \Delta(a_{12}b_{22}x_{22}) + \Delta(a_{21}b_{12}x_{22}) + \Delta(a_{22}b_{22}x_{22}) \\&=& 
\Delta(a_{11}b_{12})x_{22} + a_{11}b_{12}\Xi(x_{22}) + \Delta(a_{12}b_{22})x_{22} +  a_{12}b_{22}\Xi(x_{22}) \\&+& \Delta(a_{21}b_{12})x_{22} + a_{21}b_{12}\Xi(x_{22}) + \Delta(a_{22}b_{22})x_{22} + a_{22}b_{22}\Xi(x_{22}) \\&=& \Delta(ab)x_{22} + a_{11}b_{12}\Xi(x_{22}) +a_{12}b_{22}\Xi(x_{22}) + a_{21}b_{12}\Xi(x_{22}) \\&+& a_{22}b_{22}\Xi(x_{22}).
\end{eqnarray*} 

$\bullet$ Fourth: On the other hand,
\begin{eqnarray*}
\Delta(abx_{22}) &=& \Delta(a_{11}b_{12}x_{22}) + \Delta(a_{12}b_{22}x_{22}) + \Delta(a_{21}b_{12}x_{22}) + \Delta(a_{22}b_{22}x_{22}) \\&=& 
\Delta(a_{11})b_{12}x_{22} + a_{11}\Xi(b_{12}x_{22}) + \Delta(a_{12})b_{22}x_{22} +  a_{12}\Xi(b_{22}x_{22}) \\&+& \Delta(a_{21})b_{12}x_{22} + a_{21}\Xi(b_{12}x_{22}) + \Delta(a_{22})b_{22}x_{22} + a_{22}\Xi(b_{22}x_{22}) \\&=& \Delta(a)bx_{22} + a_{11}\Xi(b_{12}x_{22}) +a_{12}\Xi(b_{22}x_{22}) + a_{21}\Xi(b_{12}x_{22}) \\&+& a_{22}\Xi(b_{22}x_{22}) \\&=& \Delta(a)bx_{22} + a_{11}\Xi(b_{12})x_{22} + a_{11}b_{12}\Xi(x_{22}) +a_{12}\Xi(b_{22})x_{22} \\&+& a_{12}b_{22}\Xi(x_{22}) + a_{21}\Xi(b_{12})x_{22} + a_{21}b_{12}\Xi(x_{22}) + a_{22}\Xi(b_{22})x_{22} \\&+& a_{22}b_{22}\Xi(x_{22}) \\&=& \Delta(a)bx_{22} + a\Xi(b)x_{22} + a_{11}b_{12}\Xi(x_{22}) + a_{12}b_{22}\Xi(x_{22}) \\&+& a_{21}b_{12}\Xi(x_{22}) + a_{22}b_{22}\Xi(x_{22}). 
\end{eqnarray*}
Thus comparing the above two equations, we obtain $(\Delta(ab) - \Delta(a)b - a\Xi(b))x_{22} = 0$ for any $x_{22} \in \R_{22}$ then $e_1(\Delta(ab) - \Delta(a)b - a\Xi(b))e_2 = 0 = e_2(\Delta(ab) - \Delta(a)b - a\Xi(b))e_2$ by condition ($\clubsuit$). 
Therefore $\Delta(ab) = \Delta(a)b + a\Xi(b)$.
\end{proof}

\noindent {\bf Proof of Theorem \ref{main}:} 
From the above lemmas, we have proved that $\Delta : \R \rightarrow \R$ is an additive generalized derivation. Since $\Delta(a) = \delta(a) - d_{s}(a)$ for each $a \in \R$, by a simple calculation, we see that $\delta$ is also an additive generalized derivation. The proof is completed.

%\section{Generalized Jordan derivation on $M_2(\mathbb{C})$}

\begin{cor}
Let $M_2(\mathbb{C})$ denote the algebra of all $2\times 2$ complex matrices and
$B$ be an arbitrary algebra over the complex field $\mathbb{C}$. Suppose that $\delta : M_2(\mathbb{C}) \rightarrow B$ is a linear mapping such that $\delta(E^2) = \delta(E)E + E\tau (E)$ holds for all idempotent $E$ in $M_2(\mathbb{C})$, where $\tau : M_2(\mathbb{C}) \rightarrow B$ is a linear mapping satisfying $\tau(E) =
\tau(E)E + E\tau(E)$ for any idempotent $E$ in $M_2(\mathbb{C})$, then $\delta$ is a generalized derivation.
\end{cor}
\begin{proof}
Let $M_2(\mathbb{C}) = E_1 M_2(\mathbb{C}) E_1 \oplus E_1 M_2(\mathbb{C}) E_2 \oplus E_2 M_2(\mathbb{C}) E_1 \oplus E_2 M_2(\mathbb{C}) E_2$ the Peirce decomposition relative to the idempotent $E_1 =\left[
\begin{array}{c c}
1& 0\\
0& 0\\
\end{array}\right]$. Clearly $M_2(\mathbb{C})$ is semiprime and satisfies ($\spadesuit$), ($\clubsuit$).
By Theorem $4.1$ in \cite{Jing} we have $\tau$ is a derivation and $\delta(A^2) = \delta(A)A + A\tau(A)$ for any $A \in M_2(\mathbb{C})$. Therefore, by Theorem \ref{main}, $\delta$ is a generalized derivation.
\end{proof}

\begin{defn}
Let $U(\R)$ be the group of units of $\R$. An ideal $I$ of a ring $\R$ is unit-prime if for any $a, b \in \R$, $aU(\R)b \subseteq I$ implies $a \in I$ or $b \in I$, and, unit-semiprime if for any $a \in \R$, $aU(\R)a \subseteq I$ implies $a \in I$. A ring $\R$ is unit-(semi)prime if $\left(0\right)$ is a unit-(semi)prime ideal of $\R$.
\end{defn}

\begin{thm}
Matrix ring over unit-semiprime ring are unit-semiprime.
\end{thm}

\begin{proof}
See Theorem $11$ in \cite{grigore}
\end{proof}

\begin{cor}
Let $M_2$ be a $2\times 2$ matrix ring over unit-semiprime ring. Suppose that $\delta : M_2 \rightarrow M_2$ is a linear mapping such that $\delta(A^2) = \delta(A)A + A\tau (A)$ holds for all $A$ in $M_2$, where $\tau : M_2 \rightarrow M_2$ is a linear mapping satisfying $\tau(A^2) =
\tau(A)A + A\tau(A)$ for any $A$ in $M_2$, then $\delta$ is a generalized derivation.
\end{cor}

In \cite{Jing}, the authors introduced the concept of generalized Jordan triple derivation. Let $\R$ be a
ring and $\delta : \R \rightarrow \R$ an additive map. If there is a Jordan triple derivation $\tau : \R \rightarrow \R$ such that
$\delta(aba) = \delta(a)ba + a\tau(b)a + ab\tau(a)$ for every $a,b \in \R$, then $\delta$ is called a generalized Jordan triple derivation, and $\tau$ is the relating
Jordan triple derivation. Recall that $\tau$ is a Jordan triple derivation if $\tau(aba) = \tau(a)ba + a\tau(b)a + ab\tau(a)$ for any $a,b \in \R$.

The authors conjecture that every generalized Jordan triple derivation on $2$-torsion free semiprime ring is a generalized derivation.
In our case it follows the following corollary.

\begin{cor}
Let $\R$ be a $2$-torsion free semiprime unity ring satisfying $(\spadesuit)$, $(\clubsuit)$ and $\delta$ be an additive generalized Jordan triple
derivation from $\R$ into itself. If there exist an idempotent $e$ so that $e \neq 0$, $e \neq 1$ in $\R$, then $\delta$ is an additive generalized derivation.
\end{cor}
\begin{proof}
Let $\delta : \R \rightarrow \R$ be an additive generalized Jordan triple derivation and $\tau : \R \rightarrow \R$ the relating Jordan triple derivation. Note that $\tau(e_1 + e_2) = 0$, so $\tau$ is in fact a Jordan derivation. Now it is easy to check that a generalized Jordan triple derivation on $\R$ is a generalized Jordan derivation. 
Therefore, by Theorem \ref{main}, $\delta$ is an additive generalized derivation.
\end{proof}

\end{document}